\documentclass[12pt]{amsart}

\usepackage{geometry}                
\usepackage{graphicx,txfonts}
\usepackage{amssymb}

\newtheorem{theorem}{Theorem}[section]

\newtheorem{lemma}[theorem]{Lemma}

\newtheorem*{theoremA}{Theorem A}
\newtheorem*{theoremB}{Theorem B}

\newtheorem{corollary}[theorem]{Corollary}

\theoremstyle{definition}

\theoremstyle{remark}
\newtheorem{remark}[theorem]{Remark}

\numberwithin{equation}{section}

\newcommand{\cS}{\mathcal{S}}

\newcommand{\cA}{\mathcal{A}}

\newcommand{\D}{\mathbb{D}}

\newcommand{\C}{\mathbb{C}}

\newcommand{\ga}{\gamma }

\newcommand{\Om}{\Omega }

\newcommand{\diam}{\operatorname{Diam}}
\newcommand{\ndiam}{\operatorname{n-Diam}}

\newcommand{\rea}{\operatorname{Re}}

\newcommand{\area}{\operatorname{Area}}

\newcommand{\bd}[1]{\partial #1}

\newcommand{\Mod}{\operatorname{Mod}}
\newcommand{\capa}{\operatorname{Cap}}
\newcommand{\rad}{\operatorname{Rad}}

\newcommand{\Ind}{\operatorname{Ind}}

\newcommand{\defeq}{\mathrel{\mathop:}=}

\title{Some remarks about analytic functions defined on an annulus}
\author[Poggi-Corradini]{Pietro Poggi-Corradini}
\address{Department of Mathematics, Cardwell Hall, Kansas State University,
Manhattan, KS 66506, USA}
\email{pietro@math.ksu.edu}

\thanks{These note are inspired by the recent paper \cite{iwaniec-kovalev-onninen}}
\begin{document}
\begin{abstract}
Some notes and observations on analytic functions defined on an annulus
\end{abstract}

\maketitle
\baselineskip=24pt

\section{Introduction}

In \cite{bmmpcr2008}, given
$f:\D\defeq \{z\in \C:|z|<1 \}\rightarrow \C$ analytic, $f (0)=0$, we considered different ways
of measuring $f (r\D)$ for $0<r<1$.

\begin{itemize}
\item Maximum modulus: $\rad f (r\D)=\max_{|z|\leq r}|f (z)|$ (Schwarz)

\item Diameter: $\diam f (r\D)$ (Landau-Toeplitz)

\item  $n$-Diameter: $\ndiam f (r\D)$

\item  Capacity: $\capa f (r\D)$

\item  Area: $\area f (r\D)$

\item  What else? Perimeter, eigenvalues of the Laplacian, etc\dots
\end{itemize}
Note that if $f^{\prime} (0)\neq 0$, then $f$ is univalent on
$|z|<r_{0}$ for some $r_{0}>0$.

Let $M (f (r\D))$ be a measurement as above. Define:
\[
\phi_{M} (r)\defeq \frac{M (f (r\D))}{M (r\D)}.
\]
\begin{theoremA}[\cite{bmmpcr2008}] Let $M$ be Radius, $n$-Diam,
or Capacity. Then $\phi_{M} (r)$ is increasing and its $\log$ is a convex
function of $\log r$. Actually, it is
strictly increasing unless $f$ is linear. In particular,
\[
M (f (r\D))\leq \phi_{M} (R) M(r\D)\qquad 0<r<R
\]
When $M$ is Area ($f$ not univalent) ``logconvexity'' might fail but
strict monotonicity persists.
\end{theoremA}

Given a Jordan curve $J$ let $G_{1}$ be its interior and
 $G_{2}$ be its exterior in $\C \cup\{\infty \}$, and assume that $0\in G_{1}$.
Compute the reduced modulus (defined below) $M_{1}$ of $J$ with
respect to $0$ in $G_{1}$ and the reduced modulus $M_{2}$ of $J$ with
respect to $\infty$ in $G_{2}$. Then $M_{1}+M_{2}\leq 0$ with equality
if and only if  $J$ is a circle of the form $\{|z|=r \}$. In fact, the sum $M_{1}+M_{2}$
can be thought of as a measure of how far $J$ is from being a circle.
Teichm\"uller's famous {\em Modulsatz} says that if $-\delta \leq
M_{1}+M_{2}\leq 0$, for $\delta$ sufficiently small, then the
oscillation of $J$ is controlled in the sense that, for some finite
constant $C$ (independent of $J$)
\[
\frac{\sup_J|z|}{\inf_{J}|z|}\leq 1+C\sqrt{\delta \log \frac{1}{\delta}}.
\]
See Chap. V.4 in \cite{garnett-marshall2005}.

Even though a definition of reduced modulus can be given using modulus of path families,
it turns out that
\[
M_{1}=\frac{1}{2\pi}\log |g_{1}^{\prime} (0)|, \qquad M_{2}=\frac{1}{2\pi}\log \left|\left(\frac{1}{g_{2}} \right)^{\prime} (0)
\right|
\]
where $g_{1}$ is a conformal map of $\D$ onto $G_{1}$ with $g_{1} (0)=0$
and $g_{2}$ is a conformal map of $\D$ onto $G_{2}$ with $g_{2}(0)=\infty$.
Note also that the conformal map $\psi \defeq 1/g_{2}^{-1}$ sends
$G_{2}$ to $\{|z|>1 \}\cup\{\infty \}$ and $\psi (\infty)=\infty$, so
$M_{2}$ is related to the usual concept of logarithmic capacity of $J$:
\begin{equation}\label{eq:m2}
M_{2}=-\frac{1}{2\pi}\log \capa (J).
\end{equation}
In fact, letting $1/J=\{z:1/z\in J \}$, we find that
\begin{equation}\label{eq:formula}
- (M_{1}+M_{2})=\frac{1}{2\pi}\log \left(\capa (J)\capa (1/J) \right)\geq 0.
\end{equation}

Now consider a conformal map $f$ on $\D$ with $f (0)=0$ and let $J
(r)=f (r\bd\D)$. Let $T (r)=- (M_{1} (r)+M_{2} (r))$, as above, measure how
different $J (r)$ is from a circle centered at the origin. Note that
$g_{1}(z)\defeq f (zr)$ maps $\D$ conformally onto the interior of $J
(r)$, so $M_{1} (r)= (1/ (2\pi))\log r|f^{\prime} (0)|$. On the other
hand, see for instance Example III.1.1 in \cite{garnett-marshall2005},
$M_{2} (r)= -(1/ (2\pi))\log \capa (f (r\D))$. Thus we have
\[
T (r)=\frac{1}{2\pi}\log \frac{\capa f (r\D)}{|f^{\prime} (0)|\capa
r\D}.
\]

By Theorem~A, we now know that $T (r)$, as defined above, is
increasing and convex, in fact strictly increasing unless $f$ is linear.

The proof of this Theorem~A consisted in establishing the ``increasing'' and
``log-convexity'' part first and then deducing the ``Moreover part'' from
the behavior of $\area f (r\D)$ and P\'olya's inequality relating area
and capacity. More specifically, P\'olya's inequality gives the
following relation:
\[
\phi_{\area} (r)\leq \phi_{\capa}^{2} (r).
\]
To show that $\phi_{\capa}$ is strictly increasing, by log
convexity it's enough to show that $\phi_{\area} (r)$ is strictly
increasing at $0$.

After \cite{bmmpcr2008} appeared in print, C.~Pommerenke pointed out that in
\cite{pom1961} he had already shown the following result:
\begin{theoremB}
 If f is one-to-one and analytic in the annulus $\{a < |z| < b\}$,
then $\log\capa\{ f(r\bd\D) \}$ when expressed as a function of $\log
r$ is a convex function for $a<r<b$.
\end{theoremB}
For the reader's convenience we replicate Pommerenke's proof here.
\begin{proof}[Proof of Theorem~B]
Fix $r$ as above and find $w_{1},\dotsc ,w_{n}\in \bd\D$ so that
\[
\ndiam f (r\bd\D)=\prod_{j<k} |f (w_{j}r)-f (w_{k}r)|^{2/ (n (n-1))}.
\]
Then consider the function
\[
H (z)\defeq \prod_{j<k} (f (w_{j}z)-f (w_{k} (z)))
\]
which is analytic in the annulus $\{a<|z|<b \}$.

Fix $a<r_{1}<r<r_{2}<b$. Then by Hadamard's three-circles theorem,
\[
\log \max_{|z|=r}|H (z)|\leq \frac{\log (r_{2}/r)}{\log
(r_{2}/r_{1})}\log \max_{|z|=r_{1}}|H (z)| +\frac{\log (r/r_{1})}{\log
(r_{2}/r_{1})}\log \max_{|z|=r_{2}}|H (z)|,
\]
i.e.,
\[
\log \ndiam f (r\bd\D)\leq \frac{\log (r_{2}/r)}{\log
(r_{2}/r_{1})}\log \ndiam f (r_{1}\bd\D) +\frac{\log (r/r_{1})}{\log
(r_{2}/r_{1})}\log \ndiam f (r_{2}\bd\D).
\]
Now let $n$ tend to infinity.
\end{proof}

In this note we study the case of analytic functions defined on an annulus.

For $R>1$, consider the family $\cS (R)$ of analytic and one-to-one
functions $f$ that map the annulus $A (1,R)=\{1<|z|<R \}$ onto a
topological annulus $\cA$ such that the bounded component of $\C
\setminus \cA$ coincides with the unit disk $\D$, and so that as
$|z|\downarrow 1$ we have $|f (z)|\rightarrow 1$. By the Schwarz
Reflection Principle, the conformal map $f$ extends analytically across
$|z|=1$, so the family $\cS (R)$ can be defined more compactly as the
set of analytic and one-to-one maps on $A (1,R)$ such that $|f
(z)|>1$ for all $z\in A(1,R)$ and $|f (z)|=1$ for $|z|=1$.

We prove the following monotonicity result.
\begin{theorem}\label{thm:monotone}
Let $f\in \cS (R)$ and let $T (r)$ be defined as above for $1\leq
r<R$. Then, $T (r)$ is a convex function of $\log r$ and is strictly
increasing, unless $f$ is the identity.
\end{theorem}

Also, in analogy with the local case, for a function $f\in \cS (R)$ we
define the ratios:
\[
\psi_{\ndiam}(r)\defeq \frac{\ndiam (f (\cA
(1,r))\cup\overline{\D})}{\ndiam(r\D)}
\qquad\mbox{ and }\qquad
\psi_{\capa}(r)\defeq\frac{\capa(f(\cA (1,r))\cup\overline{\D})}{\capa
  (r\D)}.
\]
Then by Theorem~B, both $\log \psi_{\ndiam}$ and $\log \psi_{\capa}$ are
convex functions of $\log r$. It turns out that they are also strictly
increasing unless
$f$ is the identity.
\begin{theorem}\label{thm:monotone2}
Let $f\in \cS (R)$ and let $\psi_{\ndiam} (r)$, $\psi_{\capa} (r)$ be
defined as above for $1\leq r<R$. Then, $\log\psi_{\capa}$ is a convex
function of $\log r$ and $\psi_{\capa}$ is strictly
increasing, unless $f$ is the identity.

In particular,
\[
\capa (f (\cA (1,r))\cup\overline{\D})\leq r\frac{\capa (f (\cA
(1,R))\cup\overline{\D})}{R}.
\]
\end{theorem}
One line of proof is similar to the disk case in
that looking at the area turns out to be the crucial ingredient.
However, we show that Theorem \ref{thm:monotone2} can also be deduced
from Theorem \ref{thm:monotone}.

\section{Proof of Theorem \ref{thm:monotone}}\label{sec:proof}

By Theorem~B and (\ref{eq:formula}) we see that $T (r)$ is a convex
function of $\log r$, for $1<r<R$.
Therefore, since $T (r)\geq 0$ and $T (0)=0$, we get that
$T (r)$ is an increasing function of $r$.

Assume that $f$ is not the identity, then $f (\{|z|=r \})$ is not a circle
for $1<r<R$. For if it were, then $f$ would be a conformal map between
circular annuli and hence would be linear and hence the
identity. Therefore, by Teichm\"uller's Modulsatz, $T (r)>0$ for $1<r<R$.

Suppose $T (r)$ fails to be strictly
increasing. Then by monotonicity it would have to be constant on an
interval $[s,t]$ with $1<s<t<R$. By convexity, it would then have
to be constant and equal to $0$ on the interval $[1,t]$, but this
would yield a contradiction. So Theorem \ref{thm:monotone} is proved.

\begin{remark}\label{rem:logconv}
Note that if $F (r)=G (\log r)$ for
some convex function $G$ and $F^{\prime} (1)\geq 0$, then
$G^{\prime} (0)\geq 0$ and by convexity $G^{\prime} (t)\geq 0$ for all
$t\geq 0$, i.e., $F^{\prime} (r)\geq 0$ for all $r\geq 1$.
\end{remark}

We now turn to Theorem \ref{thm:monotone2}. First we show how it can
be deduced from Theorem \ref{thm:monotone}.

\section{Consequences of the serial rule}\label{sec:serial}

On one hand by (\ref{eq:m2}) we have
\[
\capa \left(f (\cA (1,r)\cup\overline{\D}) \right)=e^{-2\pi M_{2} (r)}.
\]
On the other hand, by the serial rule, see (V.4.1) of
\cite{garnett-marshall2005},
\[
M_{1} (r)\geq M_{1} (1)+\Mod (f (\cA (1,r))).
\]
However, $M_{1} (1)=0$ and by conformal invariance $\Mod (f (\cA
(1,r)))=\frac{1}{2\pi}\log r$. So
\[
\frac{1}{r}\leq e^{-2\pi M_{1} (r)}.
\]
Putting this together, we get
\[
\frac{\capa \left(f (\cA (1,r)\cup\overline{\D}) \right)}{r}\geq
e^{-2\pi (M_{1} (r)+M_{2} (r))}=e^{2\pi T (r)}.
\]
i.e.,
\begin{equation}\label{eq:lowerbd}
T (r)\leq \frac{1}{2\pi}\log (\psi_{\capa} (r)).
\end{equation}
Now assume that $f$ is not linear. By Teichm\"uller's Modulsatz, $T
(r)>0$ for  $1<r<R$. So by (\ref{eq:lowerbd}), $\psi_{\capa} (r)>1$ and
by Theorem~B, $\psi_{\capa} (r)$ is a convex function of $\log r$.
Therefore, we can
conclude as above that $\psi_{\capa} (r)$ is
strictly increasing.

Teichm\"uller's Modulsatz is based on the so-called
Area-Theorem. Alternatively, Theorem \ref{thm:monotone2} can be proved
using ``area'' and P\'olya's inequality, in the spirit of
\cite{bmmpcr2008}, as we will show next.

\section{From area to capacity}\label{sec:fromatoc}

Recall P\'olya's inequality:
\[
\area E\leq \pi (\capa E)^{2}.
\]
It implies that
\begin{equation}\label{eq:alc}
\psi_{\area} (r)\leq \psi_{\capa}^{2} (r)
\end{equation}
for all $1\leq r<R$.

Lemma \ref{lem:area} below will establish that
\begin{equation}\label{eq:lemar}
\psi_{\area} (\rho)>1
\end{equation}
for $1<\rho<R$, unless $f$ is linear.

Moreover $\psi_{\capa} (1)=1$. So the derivative
\[
\frac{d}{dr}_{\mid r=1}\log \psi_{\capa} (r)\geq 0.
\]
Hence, by ``convexity'', $\psi_{\capa} (r)$ is an increasing function
of $r$.

In fact, suppose $\psi_{\capa} (r)$ fails to be strictly
increasing. Then by monotonicity it would have to be constant on an
interval $[s,t]$ with $1<s<t<R$. By ``convexity'', it would then have
to be constant and equal to $1$ on the interval $[1,t]$, but this
would yield a contradiction in view of (\ref{eq:alc}) and (\ref{eq:lemar}).

So Theorem \ref{thm:monotone2} will be proved if we can establish
(\ref{eq:lemar}).

\section{Area considerations}\label{sec:areaconsid}

Each map $f \in \cS (R)$ can be expanded in a Laurent series
\[
f (z)=a_{0}+\sum_{n\neq 0}a_{n}z^{n}.
\]

The key now is to study the area function
$h (\rho )\defeq \area f (A(1,\rho))$. We use Green's theorem to
compute the area enclosed by the Jordan curve $\gamma_{\rho}(t)=f (\rho
e^{it})$, $t\in [0,2\pi]$. Thus
\[h (\rho
)+\pi=-\frac{i}{2}\int_{\ga_{\rho}}\bar{w}dw=\frac{-i}{2}\int_{0}^{2\pi}\bar{f}
(\rho e^{it})f_{\theta} (\rho e^{it})dt=
\pi\sum_{n\neq 0}n|a_{n}|^{2}\rho^{2n}.
\]
In particular, when $\rho =1$, $h (\rho)=0$, so
\begin{equation}\label{eq:prelim}
\sum_{n\neq 0}n|a_{n}|^{2}=1
\end{equation}

The following lemma can be deduced from problem 83
in \cite{polya-szego1972}.
\begin{lemma}\label{lem:area}
For all $f\in\cS (R)$, except the identity, we have for $1<\rho<R$,
\[
\area f (A (1,\rho))> \area A(1,\rho).
\]
\end{lemma}
\begin{proof}
Let $1<\rho<R$. Then, by (\ref{eq:prelim}),
\begin{eqnarray*}
h (\rho) & = & -\pi +\pi \sum_{n\neq 0}n|a_{n}|^{2} \rho^{2n}\\
& = & \pi (\rho^{2}-1)+\pi \sum_{n\neq 0}n|a_{n}|^{2} (\rho^{2n}-\rho^{2})\\
& = & \area A (1,\rho)+\pi\rho^{2}\sum_{n\neq 0}n|a_{n}|^{2} (\rho^{2n-2}-1)
\end{eqnarray*}
But $n (\rho^{2n-2}-1)\geq 0$ for all integers.
\end{proof}

This concludes the proof of Theorem \ref{thm:monotone2}.

\section{Principal frequency}\label{sec:evlapl}

Another measure for $f (r\D)$ is to consider:
\[
M_{0} (f (r\D))\defeq \frac{1}{\Lambda_{1}(f (r\D))}.
\]
Recall that given a bounded domain $\Om\subset\C$,
\[
\Lambda_{1}^{2} (\Om)=\inf \frac{\int_{\Om}|\nabla u|^{2}dA}{\int_{\Om}
u^{2}dA}
\]
where the infimum ranges over all functions $u\in C^{1}
(\overline{\Om})$ vanishing on $\bd\Om$ and is attained by a function $w\in C^{2}
(\overline{\Om})$ which is characterized as being the unique solution
to
\[
\Delta w +\Lambda^{2}w=0 , w>0 \mbox{ on $\Om$}, w=0 \mbox{ on $\bd\Om$}.
\]
It follows from \cite{polya-szego1951} p.~98 (5.8.5) that
\[
\phi_{M_{0}} (r)\left(=\frac{\Lambda_{1} (r\D)}{\Lambda_{1} (f (r\D))}
\right)>|f^{\prime} (0)|.
\]
{\bf Problem:} Show that $\phi_{M_{0}} (r)$ is strictly
increasing when $f$ is not linear.

This problem turns out to have been solved already by work of Laugesen
and Morpurgo \cite{laugesen-morpurgo1998}. Although, I'm not sure if
essentially different techniques are required in the case of the annulus.

\def\cprime{$'$}

\end{document}